\documentclass[12pt]{amsart}
\usepackage{amssymb,latexsym, amssymb, amsthm,amsmath}
\usepackage{enumerate}
\usepackage{amsfonts}
\usepackage{color}
\usepackage{comment}

\usepackage{hyperref}

\newtheorem{lemma}{Lemma}
\newtheorem{theorem}[lemma]{Theorem}

\numberwithin{equation}{section} \numberwithin{lemma}{section}

\newtheorem{definition}[lemma]{Definition}
\newtheorem{example}[lemma]{Example}

\newtheorem*{condition}{Conditions for system (2.1)}

\makeatletter
\@namedef{subjclassname@2010}{
  \textup{2010} Mathematics Subject Classification}
\makeatother

\numberwithin{equation}{section}

\newcommand{\N}{\mathbb{N}}
\newcommand{\R}{\mathbb{R}}

\newcommand{\Z}{\mathbb{Z}}
\newcommand{\Q}{\mathbb{Q}}

\newcommand{\T}{\mathbb{T}}

\newcommand{\vektor}[1]{\mathbf{#1}}
\newcommand{\av}{\vektor{a}}
\newcommand{\bv}{\vektor{b}}
\newcommand{\ev}{\vektor{e}}

\newcommand{\nv}{\vektor{n}}

\newcommand{\mv}{\vektor{m}}
\newcommand{\xv}{\vektor{x}}
\newcommand{\yv}{\vektor{y}}
\newcommand{\jv}{\vektor{j}}
\newcommand{\hv}{\vektor{h}}

\newcommand{\vv}{\vektor{v}}
\newcommand{\wv}{\vektor{w}}

\newcommand{\nullv}{\boldsymbol{0}}
\newcommand{\einsv}{\boldsymbol{1}}

\newcommand{\ma}{{\mathfrak m}}
\newcommand{\Ma}{{\mathfrak M}}

\newenvironment{blau}{\color{blue}}{}
\newenvironment{gruen}{\color{green}}{}
\newenvironment{rot}{\color{red}}{}
\specialcomment{forme}{\begin{blau}}{\end{blau}}
\excludecomment{forme}
\specialcomment{exam}{\begin{gruen}}{\end{gruen}}
\excludecomment{exam}
\specialcomment{expl}{\begin{rot}}{\end{rot}}
\excludecomment{expl}

\begin{document}

\title[Translation invariant quadratic forms]{Some refinements for translation invariant \\ quadratic forms in dense sets}

\author[Eugen Keil]{Eugen Keil}

\address{Mathematical Institute \\
University of Oxford \\
Andrew Wiles Building \\
Radcliffe Observatory Quarter \\
Woodstock Road \\
Oxford \\
OX2 6GG \\
United Kingdom}

\email{Eugen.Keil@maths.ox.ac.uk}

\date{\today}


\maketitle

\begin{abstract}
We improve the result of our previous paper on translation invariant quadratic forms 
in two special cases. We reduce the density bound 
$|\mathcal{A}|/N = O((\log\log N)^{-c})$ to $|\mathcal{A}|/N = O((\log N)^{-c})$
for most quadratic forms and handle
almost diagonal equations in $s \geq 6$ variables instead of $s \geq 8$.
\end{abstract}

\section{Introduction}

This paper complements the author's work \cite{Keil2} on 
translation invariant quadratic forms and should be read subsequent to it.
In \cite{Keil2} we have proven the following result.

\begin{theorem} \label{Thm1}
Let $Q \in \Z^{s \times s}$ be symmetric with $Q \cdot \einsv = \nullv$ and off-rank~$r$.
Assume that $\xv^TQ\xv = 0$ has a non-singular real solution and assume that 
$s \geq 5 + 3r$ for $1 \leq r \leq 4$ and $s \geq 10$ for $r \geq 5$. 
If there are only trivial solutions, when the variables
are restricted to $\mathcal{A} \subset \{1,2,\ldots,N\}$, 
then $|\mathcal{A}| \leq C_Q N (\log\log N)^{-c}$
for some $c, C_Q > 0$ and $c$ independent of $Q$.
\end{theorem}

The condition $Q \cdot \einsv = \nullv$ is equivalent to {\em translation invariance}
of the equation $\xv^TQ\xv = 0$ as explained in \cite{Keil2}. 
The {\em `off-rank'} of a matrix is defined in Section \ref{Sec-setup}.

The aim of this note is to give improvements on this theorem in two natural
special cases. On the one hand, we can improve the bound 
$|\mathcal{A}| \leq C_Q N (\log\log N)^{-c}$ to $|\mathcal{A}| \leq C_Q N (\log N)^{-c}$
in the case $r \geq 5$ (which covers almost all quadratic forms) 
and on the other hand, solve the problem with as few
as $s \geq 6$ variables in the almost diagonal situation $r=1$.
In the end, we want to discuss how the two approaches might be useful to improve on 
Theorem \ref{Thm1} for all quadratic forms.

The next section is devoted to explain the results, which are given
in Theorem \ref{Thm2} and Theorem \ref{R1-thm}.

We introduce some standard notation. 
We use the Vinogradov notation $\ll$ and $O$-notation throughout the paper
and indicate dependencies on parameters by subscripts, like in $O_{P,\epsilon}(N)$.
The asymptotic parameter $N \in \N$ should be thought of as large and we work
most of the time on the interval $[1,N] = \{1,2,\ldots,N\}$. We write
$\T = \R/\Z$ for the `circle' and identify it with $[0,1]$ whenever convenient.
As usual we abbreviate $e(x)=\exp(2\pi i x)$.

For a set $\mathcal{A} \subset [1,N]$ with the indicator function $1_{\mathcal{A}}$ 
and density $\delta =|\mathcal{A}|/N$  we define the `balanced function' by
\begin{align} \label{balanced}
f(n) = 1_{\mathcal{A}}(n) - \delta.
\end{align}

We are using bold face to denote vectors and the notation $\xv \leq N$ means
that each coordinate is bounded by $N$.\\

{\bf Acknowledgements:}\\
This work has a non-empty intersection with a chapter from the author's Ph.D. thesis. 
He would like to thank Trevor Wooley for his encouragement and helpful discussions.
The Ph.D.-studies of the author were partially supported by the EPSRC.
This work was finished at the University of Oxford with the 
support by EPSRC, grant EP/J009458/1.

\section{Statement of Results} \label{Sec-setup}

First we summarize some of the definitions from \cite{Keil2} that are needed for the proofs below.

A quadratic form $Q(\xv) = \xv^TQ\xv$ is translation invariant if $Q(\xv + \einsv) = Q(\xv)$
for all $\xv \in \Z^s$. For a corresponding symmetric matrix $Q \in \Z^{s \times s}$, this translates into the condition
$Q \cdot \einsv = \nullv$.

\begin{definition}[Off-diagonal rank]
For a symmetric matrix $Q \in \R^{s \times s}$ we consider the set of matrices $M$
such that $M = P^TQP$ for a permutation matrix $P$. Such a matrix can be (non-uniquely) written as
\begin{align*}
M = \begin{pmatrix} 
  A  & B\\ 
  B^T & C  
\end{pmatrix}
\end{align*}
for some matrices $A,B$ and $C$. The \emph{off-rank} $r$ of $Q$ is defined as 
\begin{align*}
r = \max\mbox{rank}(B),
\end{align*}
where the maximum is taken over all choices of $P$ and decompositions of $M$.
In other words, $r$ is the maximal rank of a submatrix in $Q$, that doesn't
contain any diagonal elements.
\end{definition}

Now we can state the first theorem.

\begin{theorem} \label{Thm2}
Let $Q \in \Z^{s \times s}$ be symmetric with $Q \cdot \einsv = \nullv$ and off-rank $r \geq 5$
(this implies that $s \geq 10$).
Assume that it has a non-singular real solution to $\xv^TQ\xv=0$, 
but only trivial solutions when the variables
are restricted to $\mathcal{A} \subset [1,N]$. 
Then $|\mathcal{A}| \leq C_Q N (\log N)^{-c}$
for some $c,C_Q > 0$.
\end{theorem}

The main idea for the proof can be summarized as follows. The usual density increment
procedure depends on the fact that $e(\alpha p(x))$ for a polynomial $p$ is constant
on long arithmetic progressions. If $\alpha$ is close to a rational number
with small denominator (on the `major arcs'), the length of those
progressions is comparable to $N$ and we get a much better bound. It turns out that
we can control the contribution of the `minor arcs' (complement of the major arcs) 
in the proof of Theorem \ref{Thm2} in a uniform way 
independent of the structure of $\mathcal{A}$. 
This is done in Section \ref{Sec-min}. In Section \ref{Sec-major} we 
perform the density increment with $\alpha$ in the major arcs.\\

To motivate the second part of the paper, 
consider quadratic forms defined by the following matrices.

\begin{example}
Consider the matrices
\begin{align*}
\begin{pmatrix}
3 & 1 & 0 & 0 & 0\\ 
1 & 7 & -2 & -3 & 2\\
0 & -2 & 5 & 0 & 0\\ 
0 & -3 & 0 & 1 & 0\\
0 & 2 & 0 & 0 & -4
\end{pmatrix}
\qquad \mbox{ and } \qquad 
\begin{pmatrix}
0 & 2 & 2 & 2 & 1\\ 
2 & -1 & 4 & 4 & 2\\
2 & 4 & 6 & 4 & 2\\ 
2 & 4 & 4 & -5 & 2\\
1 & 2 & 2 & 2 & 3
\end{pmatrix}.
\end{align*}

The first example is `almost diagonal' in the sense that off-diagonal terms
are all concentrated in the second column/row.
In what sense is the second example almost diagonal?
We can write the second matrix as a sum of a diagonal matrix $D$ and a rank one
perturbation $R$, a matrix of the form $R = \vv \cdot \vv^T$ for the vector $\vv = (1,2,2,2,1)^T$.
While the first matrix is an example of a small `local' perturbation, the second is an
instance of a diagonal matrix with a small `global' perturbation.
\end{example}

One can easily check that the two matrices have off-rank one.
As we will see in Section \ref{Sec-structR1}, 
it turns out that those two cases exhaust the possible ways
that a symmetric matrix can have off-rank one. Armed with this classification, we can show
that Theorem \ref{Thm1} holds for $s \geq 6$ (instead of $s \geq 5+3r=8$) 
if the matrix underlying our quadratic form has off-rank one.
The result is given in the following theorem.

\begin{theorem} \label{R1-thm}
Let $Q \in \Z^{s \times s}$ be symmetric with $Q \cdot \einsv = \nullv$ and off-rank $r=1$.
Assume that $s \geq 6$, and that there is a non-singular real solution to $\xv^TQ\xv=0$. 
If there are only trivial solutions when the variables
are restricted to $\mathcal{A} \subset [1,N]$, then 
$|\mathcal{A}| \ll_Q N (\log\log N)^{-1/15}$.
\end{theorem}

Let us analyse the two cases in more detail. What does it mean that in the first example 
all off-diagonal terms are concentrated in the row/column of one variable? 
Let us assume for simplicity that the variable is $x_s$. We will see that 
by completing squares, we can remove all off-diagonal terms and 
replace each variable $x_i$ by $x_i-x_s$ (translation invariance!), 
reducing the problem to a diagonal form in five variables, 
which is well within reach of the classical approach. We deal with this case
in Section \ref{sec-param}.

In Section \ref{Sec-diag} we reduce the second case to the following diagonal system 
considered by Smith \cite{Smith} and the author \cite{Keil}.
\begin{equation} \label{eq-sys}
\begin{split} 
d_1 x_1^2 + d_2x_2^2 + \ldots + d_s x_s^2 & = 0, \\
d_1 x_1 + d_2x_2 + \ldots + d_s x_s & = 0.
\end{split}
\end{equation}
The system \eqref{eq-sys} can be handled under the following conditions.

\begin{condition}\ \\
$(i)$ $d_1 + d_2 + \ldots + d_s = 0$,\\
$(ii)$ $s \geq 7$ and $d_i \neq 0$ for all $1 \leq i \leq s$,\\
$(iii)$ there are at least two positive and at least two negative coefficients $d_i$.
\end{condition}

Condition $(i)$ encodes the translation invariance of the system.
The bound $s \geq 7$ in condition $(ii)$ will correspond to the 
condition $s \geq 6$ in Theorem \ref{R1-thm}. 
We assume that $d_i \neq 0$ since we can easily locate
non-trivial solutions otherwise.
Condition $(iii)$ is needed for the existence of a non-singular real solution 
(see \cite{Keil} for details).

In Section \ref{Sec-diag} we use the following theorem from \cite{Keil}
to deduce Theorem \ref{R1-thm} in the second case.

\begin{theorem} \label{dio-thm}
Assume that the conditions above hold and system \eqref{eq-sys} has only trivial solutions 
for $x_i \in \mathcal{A} \subset [1,N]$.
Then $|\mathcal{A}| \leq CN (\log\log N)^{-1/15}$ for some constant $C$, which depends 
only on the coefficients $d_i$ of the system.
\end{theorem}

\section{Controlling the minor arcs} \label{Sec-min}

As in \cite{Keil2}, the main ingredient in the proof is a 
bilinear sum estimate for our exponential sum 
\begin{align*}
S_g(\alpha) = \sum_{\xv \leq N} g(\xv) e(\alpha Q(\xv)).
\end{align*}
Write $\alpha = a/q + \beta$ for a diophantine approximation
with $q \leq N, |\beta| \leq (qN)^{-1}$ and
\begin{align} \label{eq-K}
K(\alpha) = \Big(N \log q + \min\Big\{ \frac{N^2}{q}, \frac{|\log(|\beta| N^2)| + 1}{|\beta|q} \Big\} \Big)^{1/2}.
\end{align}
Then we can bound $S_g(\alpha)$ as in \cite{Keil2} for $|g| \leq 1$ by 
\begin{align*}
|S_g(\alpha)| \ll N^{s-r} K(\alpha)^r.
\end{align*}
As in \cite[Theorem 3]{Keil2}, we can refine the estimate if we are given 
$L^1$-bounds $\sum_{n \leq N} |g(n)| = O(\delta N)$.
As in \cite{Keil2}, we get the bound
\begin{align} \label{Sf-est2}
|S_g(\alpha)| \ll \delta^{s-10} N^{s-5} K(\alpha)^5.
\end{align}

This pointwise bound allows us to deduce a sharp $L^p$-estimate for $S_g(\alpha)$,
which is necessary for the density increment strategy to work.

If we look at the proof of this $L^p$ bound in \cite{Keil2}, 
we see that the condition we need
for the exponent of $K(\alpha)$ is that it is bigger than four.
This means that we can pull out a small power of $K(\alpha)$,
and improve the estimate on the minor arcs, where $K(\alpha)$ is small.

For an absolute constant $D > 1$ to be chosen later, define the major arcs for 
$q \leq D^4\delta^{-40}$ and $(a;q)=1$ by
\begin{align} \label{eq-Madef}
\Ma(q,a) = \{\alpha \in \T: \|\alpha - a/q\| \leq D^{8}\delta^{-80}N^{-2}\}
\end{align}
(disjoint for $N \geq D^8\delta^{-80}$) and set $\Ma$ to be the union of all those sets.
The minor arcs $\ma = \T\backslash \Ma$ is the complement.
The key point to notice is that the constants in the definition of $\Ma$ 
only depend on $\delta$ and an absolute constant $D$. The precise
numbers are less important.

By Dirichlet's approximation theorem, we have that every $\alpha \in \T$ is contained
in at least one ball defined by $\|\alpha - a/q\| \leq (qN)^{-1}$ for some $q \leq N$.
This means that we have two types of minor arcs. Those Dirichlet neighbourhoods
with $q > D^4\delta^{-40}$ and those with $q \leq D^4\delta^{-40}$ but 
$D^{8}\delta^{-80}N^{-2} < \|\alpha - a/q\| \leq (qN)^{-1}$.

If $q > D^4\delta^{-40}$, we have 
\begin{align*}
K(\alpha) \ll (N \log q)^{1/2} + \frac{N}{q^{1/2}} \ll D^{-2}\delta^{20} N.
\end{align*}
for $N \gg D^8\delta^{-80}$.

In the second case $q \leq D^4\delta^{-40}$, we split the two intervals 
$D^{8}(N^2\delta^{80})^{-1} < |\beta| \leq (qN)^{-1}$ into dyadic bits
$|\beta| \in (2^iN^{-2},2^{i+1}N^{-2}]$ for 
$\log_2(\delta^{-80}D^{8}) \leq i \leq \log_2(Nq^{-1})$.
For fixed $i$ we have
\begin{align*}
K(\alpha) & \ll (N \log q)^{1/2} + \Big(N^2\frac{i + 1}{q2^i}\Big)^{1/2}\\
& \ll (N \log q)^{1/2} + N2^{-i/4} \ll D^{-2}\delta^{20}N
\end{align*}
for $N \gg D^8\delta^{-80}$.

Now that we have good pointwise bounds, we cite \cite[Lemma 2]{Keil2}, which 
gives us a control of $K(\alpha)$ on average. For $p > 4$ we have
\begin{align*}
\int_0^1 |K(\alpha)|^p\,d\alpha \ll N^{p-2}.
\end{align*}
We apply the two bounds with \eqref{Sf-est2} and get
\begin{align*}
\int_{\ma} |S_g(\alpha)|\,d\alpha 
\ll & (\delta^{s-10} N^{s-5}) \sup_{\alpha \in \ma} |K(\alpha)|^{1/2} 
\int_{\ma} |K(\alpha)|^{9/2}\,d\alpha\\
\ll & (\delta^{s-10} N^{s-5}) (D^{-1}\delta^{10} N^{1/2}) N^{5/2}  
\ll D^{-1}\delta^s N^{s-2}.
\end{align*}

We can use this bound in the argument in \cite[Section 3]{Keil2}.
We know that 
\begin{align*}
\int_0^1 |S_{g}(\alpha)| \, d\alpha \gg \delta^s N^{s-2}
\end{align*}
for $g(\xv) = 1_{\mathcal{A}^s}(\xv) - \delta^s$ with $\delta = |\mathcal{A}|/N$
as long as $N \gg_Q \delta^{-2}$. By choosing $D$ sufficiently small, this implies that 
\begin{align*}
\int_{\Ma} |S_{g}(\alpha)| \, d\alpha \gg \delta^s N^{s-2}.
\end{align*}
By the same argument as in \cite[Section 3]{Keil2} we get the bound
\begin{align} \label{eq-malb}
\sup_{\alpha \in \Ma} |S_{f_i}(\alpha)| \gg \delta^{s+80} N^s
\end{align}
for a function $f_i$ that is a (tensor) product of functions $f, \delta$ and $1_{\mathcal{A}}$
(see \eqref{balanced}).

\section{Density increment on major arcs} \label{Sec-major}

Consider the density increment argument from \cite[Section 7]{Keil2}.
Instead of a general lower bound for the exponential sum, we have \eqref{eq-malb},
where $\alpha$ lies in the major arcs $\Ma$ defined in \eqref{eq-Madef}.
This implies that there is a $q \leq D^4\delta^{-40}$ and $1 \leq a \leq q$ 
such that $\|\alpha-a/q\| \leq D^{8}\delta^{-80}N^{-2}$.
For the following argument, $D$ is a fixed constant and will be absorbed
in the Vinogradov notation.\\

We decompose our interval $[1,N]$ into progressions modulo $q$.
On each progression we have $e(\alpha Q(\xv))=e(\beta Q(\xv))e(aQ(\xv)/q)$, where 
$\alpha= a/q + \beta$.
Since $\beta$ is very small, we can estimate
\begin{align*}
|e(\beta Q(\xv))-e(\beta Q(\yv))| \leq 2\pi |\beta| |Q(\xv)- Q(\yv)|
\ll (N^2\delta^{80})^{-1}  |Q(\xv)- Q(\yv)|.
\end{align*}
The second factor is constant on progressions with difference $q$.
If we write $\xv = q\mv + \hv$ and $\yv = q\nv + \hv$, we obtain
\begin{align*}
& |e(\alpha Q(\xv))-e(\alpha Q(\yv))|= |e(\beta Q(\xv))e(aQ(\hv)/q)-e(\beta Q(\yv))e(aQ(\hv)/q)|\\
= & |e(\beta Q(\xv))-e(\beta Q(\yv))| \ll (N^2\delta^{80})^{-1} |Q(\xv)-Q(\yv)|\\
\ll & (N^2\delta^{80})^{-1}  q^2N \|\mv-\nv\|_{\infty}
\ll N^{-1}\delta^{-160}\|\mv-\nv\|_{\infty}.
\end{align*}

If we cut our $q$-progressions $P_{\hv}$ into subprogressions 
$P_{\hv,\jv}$ of side length approximately
$P=\eta\delta^{240}N$ for some small fixed $\eta$,
our function $e(\alpha Q(\xv))$ is constant up to 
a small error term of size $O(\delta^{80})$. 
This implies
\begin{align*}
& \sum_{\xv \leq N} f_i(\xv) e(\alpha Q(\xv)) = 
\sum_{\hv \leq q} \sum_{\jv \leq (q\eta\delta^{240})^{-1}} 
\sum_{\xv \in P_{\hv,\jv}} f_i(\xv) e(\alpha Q(\xv)) \\
= & \sum_{\hv \leq q} \sum_{\jv \leq (q\eta\delta^{240})^{-1}} 
\sum_{\xv \in P_{\hv,\jv}} f_i(\xv) (c(\alpha,\hv,\jv) + \epsilon(\alpha,\xv)\delta^{80}) 
\end{align*}
for some functions $|c(\alpha,\hv,\jv)| \leq 1$ and $|\epsilon(\alpha,\xv)| < 1/2$. 
Taking absolute values and using \ref{eq-malb} with $\sum_{\xv}|f_i(\xv)| \ll \delta^s N^s$ gives
\begin{align*}
\sum_{\hv \leq q} \sum_{\jv \leq (q\eta\delta^{240})^{-1}} 
\Big|\sum_{\xv \in P_{\hv,\jv}} f_i(\xv) \Big| \gg \delta^{s+80} N^s.
\end{align*}
For at least one value of $\hv$ and $\jv$, we get
\begin{align*}
\Big|\sum_{\xv \in P_{\hv,\jv}} f_i(\xv) \Big| \gg \delta^{240s+s+80} N^s.
\end{align*}

Since $f_i$ is a product of $f, 1_{\mathcal{A}}$ and $\delta$ the 
sum on the left hand side splits into $s$ independent sums.
We estimate all factors trivially apart from one of those that contain the balanced
function $f$ (which exists by construction of $f_i$).
Write $P_{\hv,\jv} = \hv + qP_{\jv}$ with $P_{\jv} = P_{j_1} \times \ldots \times P_{j_s}$,
then we have for some $k \leq s$ the inequality
\begin{align*}
(\delta P)^{s-1} \Big|\sum_{x_k \in h_k + qP_{j_k}} f(x_k) \Big| \gg \delta^{240s+s+80} N^s
\gg \delta^{s+80} |P_{j_k}|^{s}.
\end{align*}

This implies a density increment of size $\delta \to \delta + \theta\delta^{81}$
for a small $\theta > 0$ 
and a loss of progression length of $N \to \eta\delta^{240}N$.
Now we perform a density increment argument, the details of which can be found
in \cite[Section 7]{Keil2}.

The number of steps of this iteration is bounded by $\theta^{-1}\delta^{-81}$ and 
we end up with the condition $(\eta\delta^{240})^{\theta^{-1}\delta^{-81}}N \ll C_Q D^8\delta^{-80}$
for some constant $C_Q$, which depends on the smallest non-trivial solution of $Q$ in $\Z$
(for a discussion of this, see \cite[Section 2]{Keil2}).
Taking logarithms and rearranging for $\delta$ gives us
$\delta \ll (\log N)^{-1/82}$, for example.

\section{Structure of Quadratic Forms with Off-rank One}  \label{Sec-structR1}

Now we turn our attention towards the second topic of this paper, the 
off-rank one situation. Given a symmetric matrix with off-rank one, 
what can we say about its structure?
The following lemma provides a complete answer to this question.
Write $\vv \otimes \wv := \vv \cdot \wv^T$ for $\vv,\wv \in \R^d$.

\begin{lemma} \label{Lem-Rank1}
Let $Q \in \Z^{s \times s}$ be a symmetric matrix with off-rank $r=1$. Then either
\begin{itemize}
\item[$(i)$] $Q = D + \vv \otimes \ev_j + \ev_j \otimes \vv$, where $\vv \in \Z^s$, 
$\ev_j$ is the standard basis vector for some $1 \leq j \leq s$ and $D \in \Z^{s\times s}$ is diagonal, or
\item[$(ii)$] $Q = m^{-1}(D + \vv \otimes \vv)$ with a diagonal quadric $D \in \Z^{s \times s}$, some vector $\vv \in \Z^s$
and $m \in \Z \backslash\{0\}$.
\end{itemize}
\end{lemma}

\begin{proof}
By permutating variables, if necessary, we can assume that $Q$ has the form
\begin{align*}
Q = \begin{pmatrix}
g & u & \av^T \\
u & h &  \bv^T \\
\av & \bv & C \\
\end{pmatrix}
\end{align*}
for some matrix $C \in \Z^{(s-2) \times (s-2)}$ and $u \neq 0$.
Now consider for some $i \neq j$ the off-diagonal submatrix 
\begin{align*}
\begin{pmatrix} u & a_j\\b_i & c_{ij} \end{pmatrix}.
\end{align*}
For this to have rank at most one, we must have $uc_{ij} = a_jb_i$.
If $\av = 0$ or $\bv = 0$ this implies that $c_{ij} = 0$
for all $i \neq j$, which corresponds to case (i).
Otherwise, we can assume that $\av \neq 0$ and $\bv \neq 0$.
Again by the off-rank property we deduce that 
the vectors $\av$ and $\bv$ are linearly dependent
and we have $\bv = \lambda \av$ for some $\lambda \in \Q\backslash\{0\}$.
By using $uc_{ij} = a_jb_i$ we deduce that $C = u^{-1} \bv \otimes \av + E$ for
some diagonal matrix $E$. It remains to `lift' this information to the matrix $Q$.

If we consider the structure of $Q$ modulo diagonal matrices, we have to show that 
one can choose diagonal entries $x$ and $y$ in such a way that the matrix
\begin{align*}
P = \begin{pmatrix}
x & u & \av^T \\
u & y &  \lambda \av^T \\
\av & \lambda \av & u^{-1} \lambda \av \otimes \av \\
\end{pmatrix},
\end{align*}
has the form $m^{-1} \vv \otimes \vv$ for some $m \in \Z \backslash\{0\}$ and $\vv \in \Z^s$.
There isn't any choice but to complete $P$ to
\begin{align*}
\begin{pmatrix}
u\lambda^{-1} & u & \av^T \\
u & \lambda u &  \lambda \av^T \\
\av & \lambda \av & u^{-1}\lambda \av \otimes \av \\
\end{pmatrix}
\end{align*}
and check that this matrix has indeed rank one.
To be able to write this as a tensor product of a vector $\vv$
with itself, we multiply with the common factor $m = k^2 \lambda u$. 
A suitable $k \in \N$ makes all entries to integers
and allows us to choose $\vv^T = k\, (u, \lambda u, \lambda \av^T)$.
\end{proof}

\section{The Parameter Method} \label{sec-param}

In this section we are going to analyse case (i) in Lemma \ref{Lem-Rank1}. 
By changing notation, we can assume that the `non-diagonal'
variable is $x_s$ and by completing the square we can write the quadratic 
equation $Q(\xv) = 0$ in the form
\begin{align*}
a_1(b_1x_1-c_1x_s)^2 + \ldots + a_{s-1}(b_{s-1}x_{s-1}-c_{s-1}x_s)^2  + a_sx_s^2= 0
\end{align*}
for some $a_i,b_i,c_i \in \Z$. We can assume that $b_i \neq 0$ for all $1 \leq i \leq s-1$.

Now we can simplify this equation by incorporating the restriction of translation invariance.
If we replace each $x_i$ by $x_i + 1$ we obtain the same quadratic form with an additional linear term,
which is two times
\begin{align*}
a_1(b_1x_1-c_1x_s)(b_1-c_1) + \ldots + a_{s-1}(b_{s-1}x_{s-1}-c_{s-1}x_s)(b_{s-1}-c_{s-1})  + a_sx_s,
\end{align*}
as well as the constant term 
\begin{align*}
a_1(b_1-c_1)^2 + \ldots + a_{s-1}(b_s-c_s)^2  + a_s.
\end{align*}
The linear term must be zero for all $\xv$, 
which implies that $b_i = c_i$ for all $i$ with $a_i \neq 0$. Where $a_i = 0$
we can simply assume that $b_i = c_i$ without changing the quadratic form.
Furthermore this implies that $a_s = 0$. These conditions are sufficient for the
vanishing of the constant term as well.

If we set $d_i = a_ib_i^2$ we end up with the almost diagonal form
\begin{align} \label{Shift-quad}
d_1(x_1-x_s)^2 + \ldots + d_{s-1}(x_{s-1}-x_s)^2 = 0.
\end{align}
If $d_i = 0$ for some $i$, we can easily find a non-trivial solution in $\mathcal{A}$.
Therefore, we can assume that all $d_i \neq 0$.

We view $x_s$ as a free `parameter' ranging over the set $\mathcal{A}$ 
and focus on the remaining $s-1$ variables.
We define the shifted exponential sum by
\begin{align} \label{eq-ShiftExp}
U_{\mathcal{A}-y}(\alpha) = \sum_{x \in \mathcal{A}} e(\alpha (x-y)^2).
\end{align}
The number of solutions to \eqref{Shift-quad} can be written as the Fourier integral
\begin{align*}
\sum_{x_s \in \mathcal{A}} \int_0^1 \prod_{i=1}^{s-1} U_{\mathcal{A}-x_s}(d_i \alpha) \,d \alpha.
\end{align*}
We compare this to the situation, where $x_1,\ldots,x_{s-1} \in [1,N]$.
The number of solutions in this case is given by
\begin{align*}
\sum_{x_s \in \mathcal{A}} \int_0^1 \prod_{i=1}^{s-1} U_{[1,N]-x_s}(d_i \alpha) \,d \alpha.
\end{align*}
The inner integral is bounded from below by 
\begin{align} \label{Half-N-int}
\int_0^1 \prod_{i=1}^{s-1} U_{[1,N/2]}(d_i \alpha) \,d \alpha
\end{align}
since the set $[1,N]-x_s$ covers either $[1,N/2]$ or $[-N/2,-1]$ completely and 
equation \eqref{Shift-quad} has even degree.

As long as $s-1 \geq 5$ and not all $d_i$ have the same sign, the classical circle method
gives us many integer solutions for a diagonal quadratic equation.
(See the book of Davenport \cite{Dav}, for example.)
The case, where all $d_i$ have the same sign is excluded by the assumption that the quadric has
a non-singular real solution.
This implies a lower bound of size $N^{s-3}$ for \eqref{Half-N-int} and we obtain
\begin{align*}
\sum_{x_s \in \mathcal{A}} \int_0^1 \prod_{i=1}^{s-1} U_{[1,N]-x_s}(d_i \alpha) \,d \alpha \gg \delta N \cdot N^{s-3}.
\end{align*}
Since there are only $\delta N$ trivial solutions of \eqref{Shift-quad} 
in $\mathcal{A}$ by assumption, we get
\begin{align} \label{eq-OneDim1}
\sum_{x_s \in \mathcal{A}} \int_0^1 
\Big|\prod_{i=1}^{s-1} U_{\mathcal{A}-x_s}(d_i \alpha) - \prod_{i=1}^{s-1} \delta U_{[1,N]-x_s}(d_i \alpha)\Big|\,d \alpha 
\gg \delta^{s} N^{s-2} - \delta N.
\end{align}
To deduce a correlation estimate, we replace the indicator function $1_{\mathcal{A}}$
in \eqref{eq-ShiftExp} by $1_{\mathcal{A}} = \delta 1_{[1,N]} + f$ with the balanced function $f$ from \eqref{balanced}.
An expansion of the first product in \eqref{eq-OneDim1} creates $2^{s-1}$ terms, 
where the first one is cancelled
by the second product. The remaining contributions can be bounded from above by 
a finite sum of terms of the form
\begin{align*}
\sum_{x_s \in \mathcal{A}} \int_0^1 \prod_{i=1}^{s-1} |U_{T(x_s)g_i}(d_i \alpha)|\,d \alpha 
\end{align*}
with $(T_yg)(x) = g(x-y)$ and at least one of the $g_i$ equal to $f$.
By H\"older's inequality, we can bound this expression by
\begin{align*}
\sup_{x_s,\alpha} |U_{T_{x_s}f}(\alpha)|^{1/2} \sum_{x_s \in \mathcal{A}} \int_0^1 |U_{g}(\alpha)|^{s-3/2} \,d \alpha 
\end{align*}
for some function $|g| \leq 1$ defined on $[-N,N]$ with $\sum_{|n| \leq N} |g(n)| \leq 2 \delta N$.
This implies $|U_{g}(\alpha)| \leq 2\delta N$ and the upper bound
\begin{align*}
(2\delta N)^{s-6} \sup_{x_s,\alpha} |U_{T_{x_s}f}(\alpha)|^{1/2} 
\sum_{x_s \in \mathcal{A}} \int_0^1 |U_{g}(\alpha)|^{9/2} \,d \alpha. 
\end{align*}
We get the bound $O(N^{5/2})$ for the integral from \cite[Theorem 6]{Keil2}.
This leads to the estimate
\begin{align*}
(\delta N)^{s-5} \sup_{x_s,\alpha} |U_{T_{x_s}f}(\alpha)|^{1/2} N^{5/2} 
\gg \delta^{s} N^{s-2} - \delta N.
\end{align*}
As long as $N \gg_Q \delta^{-2}$ we obtain the correlation estimate
\begin{align*}
\Big|\sum_{x} f(x) e(\alpha (x-x_s)^2)\Big| \gg \delta^{10} N
\end{align*}
for some $x_s \in [1,N]$ and $\alpha \in \T$.
By expanding the square, one can see that the left hand side is just
a quadratic exponential sum.
The density increment procedure in \cite[Section 7]{Keil2} gives us the bound
$\delta \ll (\log\log N)^{-1/11}$.

\section{The Diagonal System} \label{Sec-diag}

We are given a translation invariant quadratic form $\xv^TQ\xv =0$ in $s \geq 6$ variables
with off-rank one, which satisfies case $(ii)$ in Lemma \ref{Lem-Rank1}.

We can write the equation $Q(\xv) = 0$ in the form
\begin{align} \label{normal-quad}
d_1x_1^2 + \ldots + d_sx_s^2 + (v_1x_1 + \ldots + v_s x_s)^2 = 0,
\end{align}
where the coefficients $d_i$ are equal to the diagonal elements
in the matrix $D$ from Lemma \ref{Lem-Rank1}.
Consider $Q(\xv + \einsv) = Q(\xv) + 2 L(\xv) + Q(\einsv)$, where $L(\xv)$ is given by
\begin{align*}
d_1x_1 + \ldots + d_sx_s + (v_1x_1 + \ldots + v_s x_s)(v_1 + \ldots + v_s)
\end{align*}
and the constant term $Q(\einsv)$ is 
\begin{align*}
(d_1 + \ldots + d_s) + (v_1 + \ldots + v_s)^2.
\end{align*}
By translation invariance, the linear and constant terms have to disappear for all $\xv \in \Z$. 
This can only happen if for $n = v_1 + \ldots + v_s$ we have $d_i = -nv_i$ and $d_1 + \ldots + d_s = -n^2$.
In the special case $n = 0$ we end up with the linear equation 
\begin{align*}
v_1x_1 + \ldots + v_s x_s = 0,
\end{align*}
which is covered by the method of Roth \cite{Roth} and gives a better bound
than needed for our theorem here. 
We assume from now on that $n \neq 0$ and write equation \eqref{normal-quad} 
in the form of a system. For a new variable $h = v_1x_1 + \ldots + v_s x_s$ 
with $h \in \Z$ we obtain (after multiplication of the linear equation by $n$)
\begin{equation} \label{eq-sys2}
\begin{split} 
d_1x_1^2 + \ldots + d_s x_{s}^2 + h^2 & = 0, \\
d_1x_1 + \ldots + d_s x_{s} + nh & = 0,
\end{split}
\end{equation}
with the condition $d_1 + \ldots + d_s = -n^2$.
Now we (arbitrarily) restrict $h$ to $n\Z$, write $h = nx_0$ and $d_0 = n^2$. 
Then system \eqref{eq-sys2} reduces further into
\begin{equation*}
\begin{split} 
d_0 x_0^2 + d_1x_1^2 + \ldots + d_s x_{s}^2 = 0, \\
d_0 x_0 + d_1x_1 + \ldots + d_s x_{s} = 0,
\end{split}
\end{equation*}
with $d_0 + d_1 + \ldots + d_s = 0$. 

Now Theorem \ref{dio-thm} implies the result for the weaker system \eqref{eq-sys2}.
One should note at this point that conditions $(ii)$ and $(iii)$ for Theorem \ref{dio-thm} 
are implied by the assumptions in Theorem \ref{R1-thm} or can be assumed to be true
without loss of generality.

\section{Remarks and Open Problems}

Looking at the cases $r=1$ and $r=5$, where we can deal with equations in
$5+r$ variables, one might have a vague hope to be able to extend this result to
off-ranks $2 \leq r \leq 4$. If this were the case, Theorem \ref{Thm1} would require
only $s \geq 10$ variables instead of $s \geq 17$.
(Recent work of Zhao \cite{Zhao} could probably 
reduce the bound even further to $s \geq 9$.)

To prove such a result, we need a structure theorem along the lines of Lemma \ref{Lem-Rank1}
for quadratic forms with off-rank $2 \leq r \leq 4$.
Let us look at the next simplest case $r=2$.

When $r=1$ we have two cases. Slightly simplified, they correspond to the decompositions
$M = D + \vv \otimes \vv$ and $M = D + \ev_s \otimes \vv + \vv \otimes \ev_s$ for a diagonal 
matrix $D$. Combining the two ideas, we end up with three different structures for $r=2$:

\begin{itemize}
\item[(i)] $Q = D + \vv \otimes \vv + \wv \otimes \wv$,
\item[(ii)] $Q = D + \vv \otimes \vv + \ev_s \otimes \wv + \wv \otimes \ev_s$,
\item[(iii)] $Q = D + \ev_{s-1} \otimes \vv + \vv \otimes \ev_{s-1} + \ev_s \otimes \wv + \wv \otimes \ev_s$.
\end{itemize}

It is easy to check, that in all three cases we have indeed a matrix with off-rank at most two. 
Sadly this na\"{i}ve idea doesn't work and they don't cover
all possible cases of matrices with off-rank $r = 2$. Consider the matrix

\begin{align*}
\begin{pmatrix}
* & 1 & 1 & 0 & 1\\
1 & * & 0 & 1 & 1\\
1 & 0 & * & 1 & 0\\
0 & 1 & 1 & * & -1\\
1 & 1 & 0 & -1 & *	
\end{pmatrix},
\end{align*}
where stars mark arbitrary entries.

For it to be of type (i), we would need to find diagonal entries such that the resulting
matrix has rank two. By choosing suitable $3\times3$ matrices with only one entry missing, we can fill
in the diagonal entries easily and check that this cannot be done consistently.

It is also easy to see that option (iii) is not correct since there are non-zero entries
in more than two rows. For option (ii), we have to show that any $4\times4$ submatrix, which results
by deleting a row and corresponding column cannot be completed to have rank one. This follows
directly from the existence of off-diagonal $2\times2$ matrices of full rank in each of those cases.

It is an interesting problem, whether this a is pathological counterexample that can be
understood by adding a case (iv) to the above list or whether symmetric off-rank two
matrices don't have a simple classification.\\

Even though a complete classification seems a non-trivial task,
it is likely, that the ideas of this paper can be used to improve
slightly on the variable bounds for $r \in \{2,3,4\}$ and, therefore,
potentially reduce the overall bound of Theorem \ref{Thm1} from $s \geq 17$
to $s \geq 16$, for example.\\

Concerning the density bounds, it isn't hard to see that the proof given for 
Theorem \ref{Thm2} does generalize to any situation in \cite{Keil2}, where we
use the function $K(\alpha)$ from \eqref{eq-K} to bound the $L^p$-norm of our exponential sum.
This would take care of almost all quadratic forms with off-rank $1 \leq r \leq 4$
as well. The only cases, where this is not possible correspond to Section 8
in \cite{Keil2}, where we reduce the problem to a linear system in four equations.
Further advances in the linear theory could provide 
bounds of the form $(\log N)^{-c}$ for all quadratic forms in sufficiently many variables.


\begin{thebibliography}{HD}


\bibitem{Dav}
H. Davenport,
\newblock {Analytic methods for Diophantine equations and Diophantine inequalities. Second edition.}
\newblock{Cambridge University Press, Cambridge, 2005.}

\bibitem{Keil}
E. Keil, 
\newblock {\em On a diagonal quadric in dense variables,}
\newblock{arXiv:1306.4524.}

\bibitem{Keil2}
E. Keil, 
\newblock {\em Translation invariant quadratic forms in dense sets,}
\newblock{arXiv:1308.6680.}


\bibitem{Roth}
K. F. Roth, 
\newblock {\em On certain sets of integers,}
\newblock{ J. London Math. Soc.  28,  (1953), 104--109.}


\bibitem{Smith}
M. Smith, 
\newblock {\em On solution-free sets for simultaneous quadratic and linear equations,}
\newblock{ J. Lond. Math. Soc. (2)  79  (2009),  no. 2, 273--293.}


\bibitem{Zhao}
L. Zhao,
\newblock {\em The quadratic form in 9 prime variables,}
\newblock{arXiv:1402.3697.}

\end{thebibliography}
\end{document}